\documentclass[12pt,twoside,a4paper]{article}

\usepackage{a4,enumerate}
\usepackage[noadjust]{cite}
\usepackage{amssymb,amsmath,amsthm}
\usepackage{comment}
\usepackage{calc}
\usepackage{xcolor}

\newcommand\mytitle{Holomorphic families of forms, operators and $C_0$-semigroups}
\newcommand\lhead{H. Vogt, J. Voigt}
\newcommand\rhead{Operator functions defined by forms}
\swapnumbers
\numberwithin{equation}{section}

\newtheorem{theorem}{Theorem}[section]

\newtheorem{proposition}[theorem]{Proposition}

\theoremstyle{definition}

\newtheorem{remark}[theorem]{Remark}
\newtheorem{remarks}[theorem]{Remarks}

 \mathchardef\ordinarycolon\mathcode`\:
  \mathcode`\:=\string"8000
  \begingroup \catcode`\:=\active
    \gdef:{\mathrel{\mathop\ordinarycolon}}
  \endgroup

\newcommand\smid{\nonscript \mskip2mu plus2mu {\mid}%
\nonscript \mskip2mu plus2mu}     
\newcommand\scpr[2]{{(#1\smid#2)}}

\newcommand\ran{\operatorname{\rm ran}}
\newcommand\Arg{\mathop{\mathrm{Arg}}}
\renewcommand{\Re}{\operatorname{Re}}
\renewcommand{\Im}{\operatorname{Im}}

\newcommand{\R}{\mathbb{R}\nonscript\hskip.03em}
\newcommand{\N}{\mathbb{N}\nonscript\hskip.03em}
\newcommand{\C}{\mathbb{C}\nonscript\hskip.03em}
\newcommand{\K}{\mathbb{K}\nonscript\hskip.03em}
\newcommand\cA{\mathcal A}
\newcommand\cL{\mathcal L}

\newcommand\rma{{\rm (a) }}
\newcommand\rmb{{\rm (b) }}

\newcommand\eul{{\rm e}}
\newcommand\drfrac[2]{\frac{\raisebox{-0.1em}{$#1$}}{\raisebox{0.1em}{$#2$}}}
\newcommand\scdot{\mkern2.5mu{\cdot}\mkern2.5mu}

\makeatletter
\let\qedhere@ams\qedhere
\def\qedhere{\@ifnextchar[{\@qedhere}{\qedhere@ams}}
\def\@qedhere[#1]{\tag*{\raisebox{-#1ex}{\qedhere@ams}}}
\makeatletter
\def\env@cases{%
  \let\@ifnextchar\new@ifnextchar
  \left\lbrace
  \def\arraystretch{1.1}%
  \array{@{\,}l@{\quad}l@{}}%
}
\renewcommand\section{\@startsection {section}{1}{\z@}%
                                     {-3.25ex \@plus -1ex \@minus -.2ex}%
                                     {1.5ex \@plus.2ex}%
                                     {\normalfont\large\bfseries}}
\makeatother

\newcommand\set[2]{\bigl\{#1{;}\;#2\bigr\}}

\newcommand\dom{\mathop{\rm dom}}
\newcommand\tmo{^{-1}}

\renewcommand\le{\leqslant}

\renewcommand\ge{\geqslant}

\newcommand\sse{\subseteq}

\newcommand\slim{\mathop{\rm s\kern.08em\mbox{\rm -}lim}} 

\newcommand\abstracttext{\noindent
If $z\mapsto a_z$ is a holomorphic function with values in the sectorial forms
in a Hilbert space, then the associated operator valued 
function $z\mapsto A_z$ is resolvent holomorphic. We give a proof of this
result of Kato, on the basis of the Lax-Milgram lemma. We
also show that the $C_0$-semigroups $T_z$ generated by~$-A_z$ depend
holomorphically on~$z$.

\vspace{8pt}

\noindent
MSC 2010: 47A07, 47B44, 47D06
\vspace{2pt}

\noindent
Keywords: sectorial form, operator function, resolvent holomorphic, 
$C_0$-semigroup
}
\begin{document}

\medmuskip=4mu plus 2mu minus 3mu
\thickmuskip=5mu plus 3mu minus 1mu
\belowdisplayshortskip=\belowdisplayskip

\title\mytitle

\author{Hendrik Vogt and J\"urgen Voigt}

\date{}

\maketitle

\begin{abstract}
\abstracttext
\end{abstract}

\section*{Introduction}
\label{intro}

The main objective of this note is to present a proof of the following theorem 
connecting holomorphic dependence of forms in a Hilbert space with holomorphy 
of the associated operator function.

\begin{theorem}\label{thm-hol-dep}
Let $H$ be a complex Hilbert space, $V\sse H$ a dense subspace, and let 
$\Omega\sse\C$ be open. For each $z\in\Omega$ let $a_z$ be a closed sectorial 
form in $H$ with domain $\dom(a_z)=V$, and let $A_z$ denote the (m-sectorial) 
operator associated with $a_z$. Assume that for all $x,y\in 
V$ the function $\Omega\ni z\mapsto a_z(x,y)\in \C$ is holomorphic.

Then the 
function $\Omega\ni z\mapsto A_z$ is resolvent holomorphic, and the sectoriality
of $(A_z)_{z\in\Omega}$ is locally uniform.
\end{theorem}

This theorem is due to Kato \cite[Theorem~VII.4.2]{kat-80} and is proved there via
a representation of m-sectorial operators involving the square roots of their real parts.
We will present a proof that might be regarded as more natural;
our crucial observation is a formula expressing the 
operator associated with a form in terms of the `Lax-Milgram operator'; see 
Proposition~\ref{prop-LM} below.

A rather striking application of Theorem \ref{thm-hol-dep}, due to B.\;Simon, 
has been given in Kato \cite[Addendum]{kat-78}. In this application a Trotter 
product formula for sectorial forms is derived from the validity  of  the 
corresponding Trotter product formula for symmetric forms.

Concerning notation, we recall that a form $a$ is called 
\emph{sectorial} if there exist $\gamma\in\R$ and $C\ge0$ such that
\[
\mathopen|\Im a(u)| \le C(\Re a(u)-\gamma\|u\|^2)\qquad (u\in \dom(a)),
\]
and similarly for an operator in $H$. This means that the numerical range of 
the form (or the operator) is contained in a sector with vertex $\gamma$ and 
semi-angle $\arctan C$.

In Kato \cite[Section~VII.4.2]{kat-80}, a function $z\mapsto a_z$ as in 
Theorem~\ref{thm-hol-dep} is called \emph{holomorphic of type}~(a). 
We call a function $\Omega\ni z\mapsto A_z$
with values in the closed operators in $H$
\emph{resolvent holomorphic} if the following condition is satisfied.
For all $z_0\in\Omega$ and some (and then 
all) $\lambda\in\rho(A_{z_0})$ there exists an open 
neighbourhood $\Omega_{z_0}$ such that $\lambda\in\rho(A_z)$ for all 
$z\in\Omega_{z_0}$ and the function $\Omega_{z_0}\ni 
z\mapsto(\lambda-A_z)\tmo\in\cL(H)$ is holomorphic;
see Kato \cite[Theorem~VII.1.3]{kat-80}.

\medskip

In Section~\ref{sec-LM} we recall the Lax-Milgram lemma and present
the resulting formula mentioned above.
Section~\ref{sec-proof-mt} contains the proof of 
Theorem~\ref{thm-hol-dep}. In Section~\ref{sec-hol-dep-sg} we sketch a result 
that, in the particular context of Theorem~\ref{thm-hol-dep}, implies that the 
associated $C_0$-semigroups depend holomorphically on $z$.

\section{The Lax--Milgram lemma}
\label{sec-LM}

Let $V$ be a Hilbert space over $\K\in\{\R,\C\}$, and let $a\colon V\times 
V\to\K$ be a coercive bounded sesquilinear form, where \emph{coercive} means 
that 
there exists $\alpha>0$ such that
\[
\Re a(u)\ge \alpha\|u\|_V^2\qquad (u\in V).
\]
Let $V^*$ denote the anti-dual space of $V$, with the $V^*$-$V$-pairing denoted 
by $\langle\cdot,\cdot\rangle$. Then
\[
\langle\cA u,v\rangle := a(u,v)\qquad (u,v\in V)
\]
defines a bounded operator $\cA\colon V\to V^*$. The Lax--Milgram lemma states that
$\cA$ is an isomorphism, and $\|\cA\tmo\| \le 1/\alpha$; see 
\cite[Theorem 2.1]{lax-mil-54}, \cite[Satz 4.9]{alt-85} (for the complex 
case).

Let $H$ be a Hilbert space over $\K$,
and let $j\in\cL(V,H)$ be an injective operator with dense range.
Then 
\[
A:=\set{(x,y)\in H\times H}{\exists\mkern2mu u\in V\colon ju=x,\ a(u,v)=\scpr 
y{jv}_H\ \, (v\in V)}
\]
defines the operator $A$ \emph{associated with} $(a,j)$.
There exists $c>0$ such that $\|ju\|_H\le c\|u\|_V$ for all $v\in V$.
If 
$x\in \dom(A)$, then
there exists $u\in V$ such that $ju=x$ and $a(u,u)=\scpr{Ax}x$; hence
\[
\Re\scpr{Ax}x=\Re a(u,u)\ge\alpha\|u\|_V^2\ge 
\frac{\alpha}{c^2}\|x\|_H^2.
\]
This inequality means that $A$ is \emph{strictly accretive}; see 
Kato~\cite[Chapter V, \S3.11]{kat-78}.

\begin{proposition}\label{prop-LM}
In the situation described above the operator~$A$ is 
strictly m-accretive, and
\begin{equation}\label{eq-prop-LM}
A\tmo = j\cA\tmo k,
\end{equation}
with the canonical injection
$k\in\cL(H,V^*)$ defined by $H\ni y\mapsto \scpr y{j(\cdot)}_H\in V^*$ 
(the `anti-dual operator' of $j$).
\end{proposition}

\begin{proof}
Let $y\in H$.
Then $\scpr y{j(\cdot)}_H\in V^*$, so
by the Lax-Milgram lemma there exists $u\in V$ such that
\[
a(u,v) = \scpr y {jv}_H\qquad (v\in V),
\]
i.e., $\cA u=ky$.
By the definition of $A$, this implies that $x:=ju\in\dom(A)$ and $Ax=y$.
This shows that $y\in\ran(A)$ and $A\tmo y=x=j\cA\tmo ky$.
We conclude that $A$ is strictly m-accretive and 
that \eqref{eq-prop-LM} holds.
\end{proof}

\section{Proof of the main theorem}
\label{sec-proof-mt}

We start with a preliminary step of the proof of Theorem~\ref{thm-hol-dep}; 
this also serves to fix some 
notation. We note that for each $z\in\Omega$ there exist $\gamma_z\in\R$ and 
$C_z\ge0$ such that
\[
\mathopen|\Im a_z(u)| \le C_z(\Re a_z(u)-\gamma_z\|u\|_H^2)\qquad (u\in V).
\]
The closedness of $a_z$ means that the space $(V,\|\scdot\|_{a_z})$, with the 
norm
\[
\|u\|_{a_z}=\bigl(\Re a_z(u)+(1-\gamma_z)\|u\|_H^2\bigr)^{1/2}\qquad (u\in V),
\] 
is complete. Using that the embedding 
$(V,\|\scdot\|_{a_z})\hookrightarrow(H,\|\scdot\|_H)$ is 
continuous and applying the closed graph theorem we conclude that the norms 
$\|\scdot\|_{a_z}$ are pairwise equivalent. For notational convenience we can 
therefore 
assume that $(V,\scpr\cdot\cdot_V)$ is a Hilbert space with a norm equivalent 
to all norms $\|\scdot\|_{a_z}$. 

\begin{proof}[Proof of Theorem \ref{thm-hol-dep}]
For $z\in\Omega$ we define $\cA_z\in\cL(V,V^*)$ by
\[
\langle\cA_z u,v\rangle := a_z(u,v)\qquad(u,v\in V)
\]
and note that the hypotheses together with Kato \cite[Theorem~III.3.12]{kat-80}
yield the holomorphy of $\Omega\ni 
z\mapsto\cA_z\in\cL(V,V^*)$.

Let $z_0\in\Omega$.
Without loss of generality we assume that $z_0=0$ and 
that $a_0$ is sectorial with vertex $\gamma_0=1$; then there exists $C>0$ such 
that 
\[
\|u\|_V^2\le C\|u\|_{a_0}^2=C\Re a_0(u)\qquad(u\in V).
\]
There exists $r>0$ such that $B[0,r]\sse\Omega$ and 
$\|\cA_z-\cA_0\|\le\frac{1}{2C}$ for all $z\in B[0,r]$. 
This implies that for all $z\in B[0,r]$, $u\in V$ one has
\begin{equation}\label{eq-norm-ineq-0}
|a_z(u)-a_0(u)|\le\frac1{2C}\|u\|_V^2\le\frac12\Re a_0(u),
\end{equation}
in particular
\begin{equation}\label{eq-norm-ineq}
\Re a_z(u) \ge \frac12\Re a_0(u)\ge\frac1{2C}\|u\|_V^2.
\end{equation}

This inequality shows that $a_z$ is coercive for all 
$z\in B[0,r]$.
Therefore Proposition \ref{prop-LM} implies that $A_z$ is strictly m-accretive, 
and
\begin{equation}\label{eq-LM-hol}
A_z\tmo = j\cA_z\tmo k,
\end{equation}
where $j\colon V\hookrightarrow H$ denotes the embedding and $k$ is as in 
Proposition~\ref{prop-LM}.
The holomorphy of $z\mapsto \cA_z$ and the existence of the inverse
$\cA_z\tmo\in\cL(V^*,V)$ for all $z\in B(0,r)$
imply that $B(0,r)\ni z\mapsto\cA_z\tmo\in\cL(V^*,V)$ is holomorphic;
cf.\ \cite[bottom of p.\,365]{kat-80}.
By \eqref{eq-LM-hol}, this implies the holomorphy of
$B(0,r)\ni z\mapsto A_z\tmo\in\cL(H)$.

The inequalities \eqref{eq-norm-ineq-0} and \eqref{eq-norm-ineq} imply
\begin{equation*}
\mathopen|\Im a_z(u)| 
\le \mathopen|\Im  a_0(u)|+\frac12\Re a_0(u)
\le\Bigl(C_0+\frac12\Bigr)\Re a_0(u)\le(2C_0+1)\Re a_z(u).
\end{equation*}
This estimate shows that the form $a_z$ is sectorial with semi-angle 
$\arctan (2C_0+1)$ and  vertex $0$, for all $z\in B[0,r]$.
\end{proof}

\begin{remark}\label{rem-hol-dep}
We will show here that the equivalence of the norms $\|\scdot\|_{a_z}$ is locally 
uniform.
Note that this was not needed explicitly in the proof of Theorem~\ref{thm-hol-dep}.

Putting ourselves into the context of the proof of Theorem~\ref{thm-hol-dep} we
show the uniform equivalence of the norms on $B(0,r)$. For $z\in B(0,r)$ the 
form $a_z$ is sectorial with vertex $0$; so we will use the norm
\[
\|u\|_{a_z}= \bigl(\Re a_z(u) +\|u\|_H^2\bigr)^{1/2}\qquad (u\in V).
\]
From \eqref{eq-norm-ineq} and \eqref{eq-norm-ineq-0} we know that
\[
\frac12 \Re a_0(u) \le \Re a_z(u) \le \frac32\Re a_0(u)\qquad (u\in V)
\]
for all $z\in B(0,r)$, and this implies
\[
\frac12\|u\|_{a_0}^2\le \|u\|_{a_z}^2\le \frac32\|u\|_{a_0}^2\qquad (u\in V).
\]
\end{remark}

\section{Holomorphic dependence of $C_0$-semigroups}
\label{sec-hol-dep-sg}

In the context of Theorem~\ref{thm-hol-dep}, every operator $-A_z$ is the 
generator of a holomorphic $C_0$-semigroup $T_z$. The following theorem shows 
that the function $z\mapsto T_z$ is also holomorphic, in a suitable sense.
Note, however, that in this result no holomorphy of the semigroups is required.

\begin{theorem}\label{thm-hol-dep-sg}
Let $X$ be a complex Banach space, and let $\Omega\sse\C$ be open. For 
$z\in\Omega$ let $T_z$ be a $C_0$-semigroup on $X$, with generator $A_z$, and 
assume that there exists $\omega\in\R$ such that
\[
M:=\sup\set{\eul^{-\omega t}\|T_z(t)\|}{t\ge0,\ z\in\Omega} <\infty.
\]
Assume further that $\Omega\ni z\mapsto (\lambda-A_z)\tmo\in\cL(X)$ is 
holomorphic, for some $\lambda>\omega$.
Then

\rma the function $\Omega\ni z\mapsto T_z(\cdot)x\in C([0,t_1];X)$ is 
holomorphic for all $t_1>0$, $x\in X$,

\rmb the function $\Omega\ni z\mapsto T_z(t)\in\cL(X)$ is holomorphic for all 
$t\ge 0$.
\end{theorem}

\begin{proof}[Sketch of the proof]
Without loss of generality we assume $\omega=0$; then
\begin{equation}\label{eq-hd-sg-1}
\bigl\|(\lambda-A_z)^{-n}\bigr\|\le \frac M{\lambda^n}\qquad (n\in \N,\ 
\lambda>0).
\end{equation}
The holomorphy hypothesis implies that $\Omega\ni z\mapsto(\lambda-A_z)\tmo\in 
\cL(X)$ is holomorphic for all $\lambda>0$; see Kato \cite[Theorem~VII.1.3]{kat-80}.
The exponential formula shows that
\begin{equation}\label{eq-hd-sg-2}
T_z(t)=\slim_{n\to\infty}\Bigl(I-\drfrac tn A_z\Bigr)^{-n}\qquad (t\ge0),
\end{equation}
and the strong convergence is uniform for $t$ in bounded subsets of $[0,\infty)$;
see Pazy \cite[Theorem~I.8.3]{pazy-83}.

From \eqref{eq-hd-sg-1} and \eqref{eq-hd-sg-2} one obtains the assertions, 
using standard facts of the theory of Banach space valued holomorphic 
functions;
see \cite[Proposition A.3]{abhn-01}.
\end{proof}

\begin{remarks}
(a) In Theorem~\ref{thm-hol-dep-sg}, assume additionally that all the 
semigroups 
$T_z$ are holomorphic on a common sector 
$\Sigma_\theta:= \set{\tau\in\C}{\mathopen|\Arg\tau|<\theta}$, with some 
$\theta\in(0,\pi/2]$, and that
\[
M:=\sup\set{\eul^{-\omega\Re\tau}\|T_z(\tau)\|}{\tau\in\Sigma_\theta,\ 
z\in\Omega}<\infty,
\]
for some $\omega\in\R$. Then as above one can show that
\[
\Omega\ni z\mapsto T_z(\cdot)x\in C\bigl((\Sigma_{\theta'}\cup\{0\})\cap 
B_\C(0,r);X\bigr)
\]
is holomorphic for all $x\in X$, $\theta'\in(0,\theta)$, $r>0$.

(b) The statement presented in part (a) above is a well-established result; 
see~\cite[Theorem IX.2.6]{kat-78}. The proof given in this reference uses 
the representation of the semigroups expressed by contour integrals. 
The authors are not aware of a source in the literature for the result stated 
in Theorem~\ref{thm-hol-dep-sg}, for non-holomorphic semigroups.
\end{remarks}

{
}
\bigskip

\noindent
Hendrik Vogt\\
Fachbereich Mathematik\\
Universit\"at Bremen\\
Postfach 330 440\\
28359 Bremen, Germany\\
{\tt 
hendrik.vo\rlap{\textcolor{white}{hugo@egon}}gt@uni-\rlap{\textcolor{white}{%
hannover}}bremen.de}\\[3ex]
J\"urgen Voigt\\
Technische Universit\"at Dresden\\
Fachrichtung Mathematik\\
01062 Dresden, Germany\\
{\tt 
juer\rlap{\textcolor{white}{xxxxx}}gen.vo\rlap{\textcolor{white}{yyyyyyyyyy}}%
igt@tu-dr\rlap{\textcolor{white}{%
zzzzzzzzz}}esden.de}


\frenchspacing

\begin{thebibliography}{99}

\bibitem{alt-85}
H.\,W.\;Alt: \emph{Lineare Funktionalanalysis}. Springer-Verlag, 
Berlin, 1985.

\bibitem{abhn-01}
W.\;Arendt, C.\,J.\,K.\;Batty, M.\;Hieber and F.\;Neubrander:
\emph{Vector-valued Laplace Transforms and Cauchy Problems}.
Birkh\"auser, Basel, 2001.

\bibitem{kat-78}
T.\;Kato: \emph{Trotter's product formula for an arbitrary pair of 
self-adjoint contraction semigroups}. Topics in functional analysis (essays 
dedicated to M.\,G.\;Kre\u\i n on the occasion of his 70th birthday), pp. 
185--195, Adv. in Math. Suppl. Stud., vol.\,3, Academic Press, New 
York, 1978.

\bibitem{kat-80}
T.\;Kato: \emph{Perturbation Theory for Linear Operators}.
Corrected printing of the second edition, Springer-Verlag, Berlin,
1980.

\bibitem{lax-mil-54}
P.\,D.\;Lax and A.\,N.\;Milgram: \emph{Parabolic equations}. Contributions to 
the theory of partial differential equations (L.\,Bers, S.\,Bochner, F.\,John 
eds.), pp.\,167--190. Annals of Mathematical Studies, 
no.\,33. Princeton University Press, Princeton, NJ, 1954.

\bibitem{pazy-83}
A.\;Pazy: \emph{Semigroups of Linear Operators and Applications to Partial 
Differential Equations}. Springer-Verlag, New York, 1983.

\end{thebibliography}
\end{document}